\documentclass[12pt]{amsart}
\usepackage{amsmath,amssymb,amsthm, amscd}

\usepackage[english]{babel}
\usepackage[utf8x]{inputenc}
\usepackage[T1]{fontenc}
\usepackage{listings}

\usepackage[a4paper,top=3cm,bottom=2cm,left=3cm,right=3cm,marginparwidth=1.75cm]{geometry}

\usepackage{amsmath}
\usepackage{graphicx}
\usepackage[colorinlistoftodos]{todonotes}
\usepackage[colorlinks=true, allcolors=blue]{hyperref}

\usepackage{tikz}
\usetikzlibrary{arrows}
\tikzstyle{block}=[draw opacity=0.7,line width=1.4cm]

\usepackage{forest}
\usepackage{float}

\usepackage{xcolor}

\theoremstyle{plain}
\newtheorem{thm}{Theorem}

\newtheorem{lemma}[thm]{Lemma}

\newtheorem{conj}[thm]{Conjecture}
\newtheorem{prop}[thm]{Proposition}

\newtheorem{definition*}[thm]{Definition}

\theoremstyle{remark}
\newtheorem{remark}{Remark}

\newcommand{\PP}{\mathbb{P}} 

\newcommand{\QQ}{\mathbb{Q}} 

\newcommand{\Res}{\mathrm{Res}}  
\newcommand{\Disc}{\mathrm{Disc}}  
\newcommand{\Aut}{\mathrm{Aut}}  
\newcommand{\Gal}{\operatorname{Gal}}  

\newcommand{\Q}{\mathbb{Q}}

\newcommand{\fr}{\frac}
\newcommand{\tu}{\textup}

\newcommand{\CC}{\mathbb{C}}

\usepackage{amssymb,fge}

\title{Arboreal representations for rational maps with few critical points}
\author[Juul, Krieger, Looper, Manes, Thompson, and Walton]{Jamie Juul, Holly Krieger, Nicole Looper, Michelle Manes, Bianca Thompson, and Laura Walton}

\begin{document}
\maketitle

\begin{abstract}
Jones conjectures the arboreal representation of a degree two rational map will have finite index in the full automorphism group of a binary rooted tree except under certain conditions. We prove a version of Jones' Conjecture for quadratic and cubic polynomials assuming the $abc$-Conjecture and Vojta's Conjecture. We also exhibit a family of degree $2$ rational maps and give examples of degree $3$ polynomial maps whose arboreal representations have finite index in the appropriate group of tree automorphisms. 
\end{abstract}
\section{Introduction}
Let $K$ be a field, and fix an algebraic closure $\bar{K}$. Given a rational function $f\in K(x)$ of degree $d\ge2$, we use $f^n$ to denote the $n$-th iterate $f\circ f\circ\cdots\circ f$, and define $f^{0}(z):=z$. 

We say $\alpha\in\PP^1(K)$ is \textit{periodic} if $f^n(\alpha)=\alpha$ for some $n\geq 1;$ the smallest such $n$ is called the \textit{exact period of $\alpha$}. The point $\alpha$ is \textit{preperiodic} if some iterate $f^m(\alpha)$ is periodic. If all critical points of $f$ are preperiodic then we say the map is \textit{post-critically finite}, or PCF.

Let $K^s$ be the separable closure of $K$ in $\bar{K}$. Choose $\alpha\in K$; for the rest of this paper we make the mild assumption that, for every $n\geq 0$, the $d^n$ solutions to $f^n(z)=\alpha$ are distinct, thereby ensuring that these solutions live in the separable closure $K^s$ of $K$.

Of recent interest, as in ~\cite{jonesbostonga,ingramgalois,jonesbostonsettled,JKMT}, is the set of iterated preimages of $\alpha\in K$ under the map $f$:
 \[
\{a\in \PP^1(K^s):f^n(a)=\alpha \text{ for some } n\geq 0\}.
\]

We consider the tree whose vertices are given by the disjoint union of the solutions to the equations
\[
f^n(z)=\alpha \text{ for }n\geq 0;
\] 
thus, the vertex set of this tree is 
\[
\bigsqcup_{n\geq 0}\left\{a\in  \PP^1(K^s):f^{n}(a)=\alpha\right\}.
\]
The edge relation of the tree is given by the action of $f$; that is, we have an edge from $\beta_1$ to $\beta_2$ if $f(\beta_1) = \beta_2$. Given the assumption above, this tree of preimages is isomorphic to the \textit{infinite rooted $d$-ary tree}, which is denoted by $T_\infty.$ See Figure~\ref{treeexample}.

\begin{center}
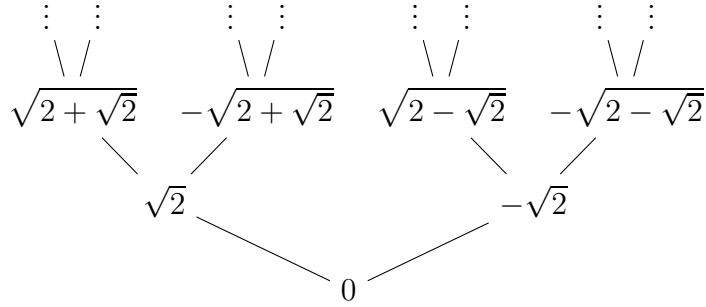
\begin{figure}[H]
\begin{tabular}{c}
\begin{forest}for tree={grow=north}
[$0$
   [$-\sqrt{2}$   
   		[$-\sqrt{2-\sqrt{2}}$
        [$\vdots$]
   		[$\vdots$
    ]]
   		[$\sqrt{2-\sqrt{2}}$
        [$\vdots$]
   		[$\vdots$
    ]
    ]]
   [$\sqrt{2}$
   		[$-\sqrt{2+\sqrt{2}}$
        [$\vdots$]
   		[$\vdots$
    ]]
   		[$\sqrt{2+\sqrt{2}}$
        [$\vdots$]
   		[$\vdots$
    ]
    ]]]
\end{forest}
\end{tabular}
\caption{The first few levels of $T_\infty$ when pulling back $\alpha=0$ by the polynomial $f(z)=z^2-2$.}
\label{treeexample}\end{figure}
\end{center}

Let $\Aut (T_\infty)$ denote the group of automorphisms of the infinite rooted tree; this group is an infinite wreath product of $S_d$, the symmetric group on $d$ letters. The absolute Galois group of $K$ acts on this copy of $T_\infty$ as tree automorphisms, which defines a continuous homomorphism $$\rho_\infty \colon \Gal(K^s/K) \to \Aut (T_\infty).$$
A continuous homomorphism $\Gal(K^s/K) \to \Aut (T_\infty)$ is called an \emph{arboreal Galois representation}~\cite[Definition~1.1]{jonesbostonga}; the particular representation $\rho_\infty$ defined above is called the \emph{arboreal Galois representation associated to the pair $(f,\alpha)$ over $K$.} The study of arboreal Galois representations dates back to work of R. W. K. Odoni in the 1980s ~\cite{odoni1,odoni3,odoni2}.  The image of $\rho_\infty$, which we denote by $G_\infty(f,\alpha)$, or $G(f)$ if $\alpha=0$, is well-studied, particularly in the degree two case \cite{HamJonMad2015,Hindes,jonessurvey,JonesManes2014}, and is the focus of this paper.

Jones conjectures for degree 2 rational maps that the following is true~\cite{jonessurvey}.
\begin{conj}\label{jonesconjecture}
Let $K$ be a global field and suppose that $f\in K(x)$ has degree two. Then $[\Aut(T_\infty):G_\infty(f,\alpha)]=\infty$ if and only if one of the following holds:
\begin{enumerate}
\item The map $f$ is post critically finite. 
\item The two critical points $\gamma_1$ and $\gamma_2$ of $f$ have a relation of the form $f^{r+1}(\gamma_1)=f^{r+1}(\gamma_2)$ for some $r\geq 1.$
\item The root $\alpha$ of $T_\infty$ is periodic under $f.$
\item There is a nontrivial M\"obius transformation that commutes with $f$ and fixes~$\alpha$.
\end{enumerate} 
\end{conj}


The `if' direction of this conjecture is already established \cite{jonessurvey}. We will prove that assuming Vojta's Conjecture for blowups of $\mathbb{P}^1\times\mathbb{P}^1$ and the $abc$-Conjecture, a similar set of conditions characterizes the set of quadratic and cubic polynomials $f\in K[x]$ such that $[\textup{Aut}(T_\infty):G(f)]<\infty$, where $K$ is a number field.

\begin{definition*}
	Let $K$ be a number field. We say $f\in K[x]$ is \textbf{eventually stable} if the number of irreducible factors of $f^n(x)$ over $K$ is bounded as $n\to\infty$.
\end{definition*}

	\begin{thm}{\label{thm:quadraticpoly}} Assume the $abc$-Conjecture for number fields. Let $K$ be a number field, and let $f\in K[x]$ have degree~$2$. Then $[\textup{Aut}(T_\infty):G(f)]=\infty$ if and only if one of the following holds:
		
	\begin{enumerate}
		\item $f$ is PCF
		\item $f^n(x)$ is not eventually stable.
	\end{enumerate}
		
	\end{thm}

	\begin{thm}{\label{thm:cubicpoly}} Assume the $abc$-Conjecture for number fields, and assume Vojta's Conjecture. Let $K$ be a number field, and let $f\in K[x]$ have degree~$3$. Then $[\textup{Aut}(T_\infty):G(f)]=\infty$ if and only if one of the following holds:
		
	\begin{enumerate}
		\item $f$ is PCF
		\item $f^n(x)$ is not eventually stable
		\item The finite critical points $\gamma_1,\gamma_2$ of $f$ have a relation of the form $f^r(\gamma_1)=f^r(\gamma_2)$ for some $r$.
	\end{enumerate}
		
	\end{thm}

\begin{remark}
	Condition (iii) includes the case $\gamma_1=\gamma_2$.
\end{remark}
    
    We use the eventual stability condition in the above theorems in place of Jones' condition on the periodicity of the root. Conjecturally, these conditions are equivalent for number fields $K$ \cite[Conjecture 1.2]{joneslevy}.
    
    In order to prove Theorems \ref{thm:quadraticpoly} and \ref{thm:cubicpoly} we first prove a set of sufficient conditions for finite index, Theorem \ref{thm:PPDsufficiency}. We then use the $abc$-Conjecture and Vojta's Conjecture to prove that these conditions are met. Another key ingredient in the degree $3$ polynomial case is a result of Huang restricting common divisors in distinct orbits \cite{huang}. These arguments will not apply in the context of Jones' original conjecture of degree 2 rational functions since the results in \cite{huang} only apply to polynomials. 
    
    The sufficient conditions of Theorem \ref{thm:PPDsufficiency} and related result can be used to find examples of cubic polynomials and quadratic rational functions with finite index. In addition to providing known examples of cubic polynomials with this property, we give new a family of degree $2$ rational maps and prove that the Galois groups have finite index for several parameters (in fact, they will have index 1).

\begin{thm}\label{thm: partialresults1}
Consider the family \[
	f_b(z)=\frac{z^2-2bz+1}{(-2+2b)z}.
	\] For parameters $b\in \mathbb{Z}$ satisfying $b\equiv 2\mod 4$ and $b>0$ or $b\equiv 4\mod 8$,  $$[\Aut(T_\infty):~G(f_b)]=1$$  
    when $K=\mathbb{Q}$ and hence $$[\Aut(T_\infty):~G(f_b)]<\infty$$
    when $K$ is any number field.
\end{thm}



\section{Toward a Serre-type open image theorem for arboreal representations}\label{section:openimage}

    
		
		
		
		


We begin by proving a set of sufficient conditions for $[\textup{Aut}(T_{\infty}):G(f)]<\infty$. We will use Capelli's Lemma and Dedekind's Discriminant Theorem in the proof.

\begin{lemma}[Capelli's Lemma]\label{lem:CapelliLemma} Let $K$ be a field, and $f(x), g(x) \in K[x]$.  Let $\alpha \in \bar{K}$ be a root of $g(x)$.  Then $g(f(x))$ is irreducible over $K$ if and only if both $g$ is irreducible over $K$ and $f(x) - \alpha$ is irreducible over $K(\alpha)$. 
\end{lemma}

\begin{thm}[\cite{Koch}, p.100]{\label{thm:Deddisc}}
	Let $K\subset L$ be number fields, with rings of integers $\mathcal{O}_K$ and $\mathcal{O}_L$ respectively. Let $\mathfrak{p}\mathcal{O}_L=\prod\mathfrak{q}_i^{e_i}$, where $f_i=f(\mathfrak{q}_i|\mathfrak{p})$ is the inertial degree of $\mathfrak{q}_i$ over $\mathfrak{p}$. Then $\mathfrak{p}$ divides $\textup{Disc}(L)$ to multiplicity at least $\sum_i (e_i-1)f_i$, with equality if, for all $i$, the residue characteristic of $\mathfrak{p}$ does not divide $e_i$.
\end{thm}

\begin{thm}{\label{thm:PPDsufficiency}}
	Let $K$ be a number field, and let $f\in K[x]$ be a monic polynomial of degree $d\ge 2$, where $d$ is prime. Suppose $f^n(x)$ has at most $r$ irreducible factors over $K$ as $n\to\infty$, so that \[f^{N+n}(x)=f_{N,1}(f^n(x))f_{N,2}(f^n(x))\cdots f_{N,r}(f^n(x))\] is the prime factorization of $f^{N+n}(x)$ in $K[x]$ for any sufficiently large $N$. Suppose that there is an $M$ such that the following holds: for each $n\ge M$, and for each $1\le j\le r$, there is a multiplicity one critical point $\gamma_1$ of $f$ and a prime $\mathfrak{p}_n$ of $K$ such that:
	
	\begin{itemize}
		\item $v_{\mathfrak{p}_n}(f_{N,j}(f^n(\gamma_1)))=1$
		\item $v_{\mathfrak{p}_n}(f^m(\gamma_1))=0$ for all $m<n+N$
		\item $v_{\mathfrak{p}_n}(f^m(\gamma_t))=0$ for all critical points $\gamma_t$ of $f$ with $\gamma_t\ne\gamma_1$ and all $m\le n+N$
		\item $v_{\mathfrak{p}_n}(d)=0.$
	\end{itemize}

Then $[\textup{Aut}(T_{\infty}):G(f)]<\infty$.\end{thm}


\begin{remark}
	Theorem \ref{thm:PPDsufficiency} does not apply to unicritical polynomials of degree at least 3, as the critical point $\gamma_1$ is required to be of multiplicity one. When $f$ is unicritical with $d\ge 3$, it is easy to see that $[\textup{Aut}(T_{\infty}):G(f)]=\infty$. In this case, $f$ is conjugate to $x^d+c$ and we can see that $G(f)$ is isomorphic to a subgroup of the infinitely iterated wreath product of $C_d$ with itself, which has infinite index in $\Aut(T_\infty)$.
\end{remark}


We will make use of the following discriminant formulas from \cite{AitHajMai2005}. Let $\psi\in K[x]$ be of degree $d$ with leading coefficient $\alpha$. \begin{equation*}{\label{discformula}}\tu{Disc}_x(\psi(x)-t)=(-1)^{(d-1)(d-2)/2}d^d\alpha^{d-1}\prod_{b\in R_{\psi}} (t-\psi(b))^{e(b,\psi)} \end{equation*} where $R_{\psi}$ denotes the set of critical points of $\psi$, $e(b,\psi)$ denotes the multiplicity of the critical point $b$, and $t$ is in $K$. From this we obtain \begin{equation*}\tu{Disc}_x(\psi^n(x)-t)=(-1)^{(d^n-1)(d^n-2)/2}d^{nd^n}\alpha^{(d^n-1)/(d-1)}\prod_{c\in R_{{\psi}^n}}(t-\psi^n(c))^{e(c,\psi^n)}\end{equation*} which is equal to \begin{equation}{\label{eqn:disciteratesformula}}(-1)^{(d^n-1)(d^n-2)/2}d^{nd^n}\alpha^{(d^n-1)/(d-1)}\prod_{b\in R_{\psi}}\prod_{i=1}^n(t-\psi^i(b))^{e(b,\psi)}.\end{equation}

\begin{proof}[Proof of Theorem $\ref{thm:PPDsufficiency}$]
	For all $n\ge N$, let $S_{n,j}$ be the set of roots of $f^{N+n}(x)$ whose defining polynomial over $K$ is $f_{N,j}(f^n(x))$. Assume without loss of generality that all of the $f_{N,i}$ are monic. Let $\alpha_i\in S_{n-1,j}$. Then \[N_{K(\alpha)/K}(f(\gamma_1)-\alpha_i)=f_{N,j}(f^{n-1}(f(\gamma_1)))=f_{N,j}(f^n(\gamma_1)).\] We thus see that if for all $l\ne 1$ we have
	$$v_{\mathfrak{p}_n}(d)=0,\ v_{\mathfrak{p}_n}(f_{N,j}(f^n(\gamma_1)))=1,\ v_{\mathfrak{p}_n}(f_{N,j}(f^n(\gamma_l)))=0,$$
then $v_{\mathfrak{p}}(\textup{Disc}(f(x)-\alpha_i))=1$ for some prime $\mathfrak{p}$ of $K(\alpha_i)$ lying above $\mathfrak{p}_n$. 
	
	By Lemma \ref{lem:CapelliLemma}, $f(x)-\alpha_i$ is irreducible over $K(\alpha_i)$, so $f(x)-\alpha_i$ is the defining polynomial of $K(\beta)/K(\alpha_i)$, where $\beta$ is some pre-image of $\alpha_i$ under $f$. By Theorem \ref{thm:Deddisc}, since $d$ is prime, we have $e(\mathfrak{q}\vert\mathfrak{p})=2$ for any prime $\mathfrak{q}$ of the Galois closure $M_i=K(f^{-1}(\alpha_i))$ of $K(\beta)/K(\alpha_i)$ lying above $\mathfrak{p}$. In fact, $I(\mathfrak{q}|\mathfrak{p})$ acts as a transposition on the roots of $f(x)-\alpha_i$. By a  standard theorem attributed to Jordan~\cite{Isaacsgrouptheory}, if $G$ is a primitive permutation group which is a subgroup of $S_d$ and $G$ contains a transposition then $G=S_d$. Since $d$ is prime any transitive subgroup of $S_d$ is primitive. It follows that $\Gal(M_i/K(\alpha_i))\cong S_d$.  
	
	Let $\mathfrak{Q}$ be any prime of $K_{N+n-1}M_i$ lying above $\mathfrak{q}$, and let $\mathfrak{P}=\mathfrak{Q}\cap K_{N+n-1}$. Our hypotheses on $\mathfrak{p}_n$, along with (\ref{eqn:disciteratesformula}), imply that $\mathfrak{p}_n$ does not ramify in $K_{N+n-1}$. We thus have $e(\mathfrak{P}|\mathfrak{p})=1$, which forces $e(\mathfrak{Q}|\mathfrak{q})=1$ as well. Hence $e(\mathfrak{Q}|\mathfrak{P})=2$. Since any non-trivial element of $I(\mathfrak{Q}|\mathfrak{P})$ descends to a non-trivial element of $I(\mathfrak{q}|\mathfrak{p})$ with the same action on the roots of $f(x)-\alpha_i$, the non-trivial element of $I(\mathfrak{Q}|\mathfrak{P})$ must act as a transposition on these roots. As $K_{N+n-1}$ is a Galois extension of $K(\alpha_i)$, $\Gal(K_{N+n-1}M_i/K_{N+n-1})$ is a normal subgroup of $\Gal(M_i/K(\alpha_i))\cong S_d$. But a normal subgroup of $S_d$ containing a transposition must be $S_d$, so we conclude that 
$$\Gal(K_{N+n-1}M_i/K_{N+n-1})\cong S_d.$$ 
	
	Now let $\widehat{M_i}=K_{N+n-1}\prod_{j\ne i}M_j$. Let $\mathfrak{Q}'$ be a prime of $K_n$ lying above $\mathfrak{Q}$, and let $\mathfrak{P}'$ be the prime of $\widehat{M_i}$ lying below $\mathfrak{Q}'$. We know that $\mathfrak{P}$ cannot divide $\textup{Disc}(f(x)-\alpha_k)$ for any root $\alpha_k$ of $f^{N+n-1}(x)$ with $k\ne i$; otherwise, $\mathfrak{P}$ divides either $f(\gamma_1)-\alpha_k$ or $f(\gamma_t)-\alpha_k$ for some $t\ne1$. In the former case, $\mathfrak{P}$ then divides $\alpha_i-\alpha_k\mid\textup{Disc}(f^{N+n-1})$, contradicting the hypotheses on $\mathfrak{p}_n$. In the latter case, $\mathfrak{P}$ divides 
$$N_{K(\alpha_j)/K}(f(\gamma_t)-\alpha_k)\mid f^{n+N}(\gamma_t),$$ 
also contradicting our hypotheses on $\mathfrak{p}_n$. Therefore $e(\mathfrak{P}'|\mathfrak{P})=1$. This forces $e(\mathfrak{Q}'|\mathfrak{Q})=1$, and so $e(\mathfrak{Q}'|\mathfrak{P}')=2$. 

	By a similar argument as above, we conclude that $\Gal(K_{N+n}/\widehat{M_i})$ contains a transposition, as it is a normal subgroup of $\Gal(K_{N+n-1}M_i/K_{N+n-1})\cong S_d$. We obtain 
$$\Gal(K_{n+N}/\widehat{M_i})\cong S_d,$$
and so $\Gal(K_{n+N}/K_{N+n-1})\cong S_d^m$, where $m=\textup{deg}(f^{N+n-1})=d^{N+n-1}$. Since this holds for all sufficiently large $n$, $[\textup{Aut}(T_\infty):G(f)]<\infty$. \end{proof}

We will also make use of the $abc$-Conjecture in the proofs of Theorem \ref{thm:quadraticpoly} and Theorem \ref{thm:cubicpoly}. For $(z_1,\dots,z_n)\in K^n\backslash \{(0,\dots,0)\}$ with $n\ge 2$, let $N_\mathfrak{p}=\fr{\log(\#k_\mathfrak{p})}{[K:\QQ]}$, where $k_\mathfrak{p}$ is the residue field of $\mathfrak{p}$, we define the \textit{height} of the $n$-tuple $(z_1,\dots,z_n)$ by 
\begin{align*}
h(z_1,\dots,z_n)&=\sum_{\textup{primes }\mathfrak{p} \textup{ of } \mathcal{O}_K} -\min\{v_{\mathfrak{p}}(z_1),\dots,v_{\mathfrak{p}}(z_n)\}N_{\mathfrak{p}}\\
&~+\dfrac{1}{[K:\QQ]} \sum_{\sigma:K\hookrightarrow\CC} \max\{\textup{log}|\sigma(z_1)|,\dots,\textup{log}|\sigma(z_n)|\}.
\end{align*} For any $(z_1,\dots,z_n)\in (K^*)^n$, $n\ge 2$, we define \[I(z_1,\dots,z_n)=\{\textup{primes }\mathfrak{p} \textup{ of } \mathcal{O}_K\mid v_{\mathfrak{p}}(z_i)\ne v_{\mathfrak{p}}(z_j)\textup{ for some } 1\le i,j\le n\}\] and let \[\textup{rad}(z_1,\dots,z_n)=\sum_{\mathfrak{p}\in I(z_1,\dots,z_n)} N_{\mathfrak{p}}.\] 

\begin{conj}[$abc$-Conjecture for number fields] Let $K$ be a number field. For any $\epsilon>0$, there exists a constant $C_{K,\epsilon}>0$ such that for all $a,b,c\in K^*$ satisfying $a+b=c$, we have \[h(a,b,c)< (1+\epsilon)(\textup{rad}(a,b,c))+C_{K,\epsilon}.\]
\end{conj}

\begin{prop}{\label{prop:abc-GNT}} Let $K$ be a number field, and assume the $abc$-Conjecture for $K$. Let $F\in K[x]$ be a separable polynomial of degree $D\ge 3$. Then for every $\epsilon>0$, there is a constant $C_\epsilon$ such that for every $\gamma\in K$ and every $n\ge 1$ with $F(\gamma)\ne0$, \[\sum_{v_\mathfrak{p}(F(\gamma))>0} N_\mathfrak{p}\ge (D-2-\epsilon)h(\gamma)+C_\epsilon.\] \end{prop}

\begin{proof}
	See Proposition 3.4 in \cite{GraNguTuc2013}.
\end{proof}

\begin{prop}[cf.~Proposition 5.1 of \cite{GraNguTuc2013}]{\label{prop:smallimprimitive}}
	Let $K$ be a number field, and let $f\in K[x]$ be of degree $d\ge 2$. Let $\alpha\in K$ have infinite forward orbit under $f$. Let $Z$ denote the set of primes of $\mathcal{O}_K$ such that $\min(v_\mathfrak{p}(f^m(\alpha)), v_\mathfrak{p}(f^n(\alpha)))>0$ for some $m<n$ such that $f^m(\alpha)\ne 0$. Then for any $\delta>0$, there exists an integer $N$ such that \[\sum_{\mathfrak{p}\in Z} N_\mathfrak{p}\le \delta h(f^n(\alpha))\] for all $n\ge N$ with $f^n(\alpha)\ne0$.
\end{prop}

\begin{proof}
	 We have $h(f^n(\alpha))\le d^n(\hat{h}_f(\alpha)+O(1))$ and $h(f^n(0))\le d^n(\hat{h}_f(0)+O(1))$ where $\hat{h}_f$ is the canonical height, see \cite[Theorem 3.20]{ads}. If $\mathfrak{p}$ divides $f^n(\alpha)$ and $f^k(\alpha)$ for some $k<n$, then $\mathfrak{p}$ divides $f^{n-k}(0)$. Therefore, any prime divisor of $f^n(\alpha)$ that divides $f^k(\alpha)$ for some $k<n$ divides either $f^k(\alpha)$ or $f^k(0)$ for some $1\le k\le\lfloor n/2\rfloor$. This yields \begin{equation*}\begin{split}\sum_{\mathfrak{p}\in Z} N_\mathfrak{p} & \le \sum_{i=1}^{\lfloor n/2\rfloor} h(f^i(\alpha))+h(f^i(0)) \\ & \le \sum_{i=1}^{\lfloor n/2\rfloor} d^i(\hat{h}_f(0)+\hat{h}_f(\alpha)+O(1))\\ & \le d^{\lfloor n/2\rfloor +1}(\hat{h}_f(0)+\hat{h}_f(\alpha)+O(1))\\ & \le \delta h(f^n(\alpha))\end{split}\end{equation*} for all sufficiently large $n$.
\end{proof}


\begin{prop}\label{prop: primitiveprime}
Let $f(x)\in K[x]$ where $K$ is a number field. Suppose $f$ has degree $d$ and $f^n(x)$ has at most $r$ irreducible factors over $K$ as $n\to\infty$, so that \[f^{N+n}(x)=f_{N,1}(f^n(x))f_{N,2}(f^n(x))\cdots f_{N,r}(f^n(x))\] is the prime factorization of $f^{N+n}(x)$ in $K[x]$ for any sufficiently large $N$. Let $\gamma\in K$ have infinite forward orbit under $f$. Let $Z$ denote the set of primes of $\mathcal{O}_K$ such that $\min(v_\mathfrak{p}(f^m(\alpha)), v_\mathfrak{p}(f^{n+N}(\alpha)))>0$ for some $m<n+N$ such that $f^m(\alpha)\ne 0$. Fix $j$, then for all sufficiently large $n$ there is a prime $\mathfrak{p}\notin Z$ such that $v_{\mathfrak{p}}(f_{N,j}(f^n(\gamma)))=1$. 
\end{prop}

The proof of this proposition is similar to that of \cite[Theorem 1.2]{GraNguTuc2013}.

\begin{proof} 
Choose an $i$ so that $f_{N,j}(f^i(x))$ has degree $D\geq 8$. By Proposition \ref{prop:abc-GNT} with $\epsilon =1$, 
\begin{equation}\label{eqn:boundonsumNp}
\sum_{v_{\mathfrak{p}}(f_{N,j}(f^{n}(\gamma)))>0}N_{\mathfrak{p}}\geq(D-3)h(f^{n-i}(\gamma))+C_1.
\end{equation}

On the other hand, since $h(\psi(z))\leq dh(z)+O_{\psi}(1)$ for any rational function $\psi$ of degree $d$ (see~\cite[Theorem 3.11]{ads}), we also have 
\begin{equation*}
\sum_{v_{\mathfrak{p}}(f_{N,j}(f^{n}(\gamma)))>0} v_{\mathfrak{p}}(f_{N,j}(f^{n}(\gamma))) N_{\mathfrak{p}}\leq D h(f^{n-i}(\gamma))+O(1).
\end{equation*}
From this, we can see that 
\begin{equation}\label{eqn:boundonsumNp2}
\sum_{v_{\mathfrak{p}}(f_{N,j}(f^{n}(\gamma)))\geq 2}  N_{\mathfrak{p}}\leq \frac{D}{2} h(f^{n-i}(\gamma))+O(1).
\end{equation}

Combining Equations (\ref{eqn:boundonsumNp}) and (\ref{eqn:boundonsumNp2}) gives 
\[\sum_{v_{\mathfrak{p}}(f_{N,j}(f^{n}(\gamma)))=1}N_{\mathfrak{p}}>\left(\frac{D}{2}-3\right)h(f^{n-i}(\gamma))+C_2>h(f^{n-i}(\gamma))+C_2.\]

Now by Proposition \ref{prop:smallimprimitive} with $\delta = \frac{1}{2d^{N+i}}$ we see that for all sufficiently large $n$, 
\[\sum_{\mathfrak{p}\in Z} N_{\mathfrak{p}}< \frac{1}{2d^{N+i}}h(f^{N+n}(\gamma)) <\frac{1}{2} h(f^{n-i}(\gamma))+C_3.\]
Hence, if $h(f^{n-i}(\gamma))>2(C_3-C_2)$, 
\[\sum_{\mathfrak{p}\in Z} N_{\mathfrak{p}}<\sum_{v_{\mathfrak{p}}(f_{N,j}(f^{n}(\gamma)))=1}N_{\mathfrak{p}},\]
proving the result.
\end{proof}

We are now ready to prove Theorem~\ref{thm:quadraticpoly}.

\begin{proof}[Proof of Theorem~$\ref{thm:quadraticpoly}$]
We can see that each of the conditions are sufficient for $[\Aut(T_\infty):G(f)]=\infty$ by the comments after Conjecture 3.11 of \cite{jonessurvey}. 

Now suppose that $f$ is not PCF and $f$ is eventually stable. We can conjugate $f$ by scaling to give a monic polynomial. Then we can see by   Proposition~\ref{prop: primitiveprime} that the conditions of Theorem~\ref{thm:PPDsufficiency} are met. Hence, $[\Aut(T_\infty):G(f)]<\infty$.
\end{proof}

A key ingredient in our proof of the degree 3 polynomial case is the following result from \cite{huang}.
	
For $a,b\in K^*$, define the generalized greatest common divisor of $a$ and $b$ as \[h_\textup{gcd}(a,b)=\frac{1}{[K:\mathbb{Q}]}\sum_{v\in M_K^0} n_v\min(v^+(a),v^+(b))\] where $v^+(a)=\max\{-\log|a|_v,0\}$, $M_K^0$ denotes the non-archimedean places of $K$, and $n_v=[K_v:\Q_v]$.
	
	\begin{thm}[\cite{huang}]{\label{thm:Huang}}
		Assume Vojta's Conjecture. Let $K$ be a number field and $f\in K[x]$ be a polynomial of degree $d\ge 2$. Assume that $f$ is not conjugate \textup(by a rational automorphism defined over $\overline{\Q}$\textup) to a power map or a Chebyshev map. Suppose $a,b\in K$ are not exceptional for $f$. Assume that there is no polynomial $H\in\overline{\QQ}[x]$ such that: $H\circ f^k=f^k\circ H$ for some $k\ge 1$, $H(0)=0$, and $H(f^l(a))=f^m(b)$ or $H(f^l(b))=f^m(a)$ for some $l,m\ge 1$. Then for any $\epsilon>0$, there exists a $C=C(\epsilon,a,b,f)>0$ such that for all $n\ge 1$, we have \[h_\textup{gcd}(f^n(a),f^n(b))\le\epsilon d^n+C.\]  
	\end{thm}

\begin{remark}
	The condition $H(0)=0$ does not appear in \cite{huang}; it is readily seen that the conclusion of Theorem \ref{thm:Huang} holds if such an $H$ satisfies $H(0)\ne 0$, as we have taken $\alpha=\beta=0$ in the original statement of \cite{huang}.
\end{remark}

The form of Vojta's Conjecture used in this result is as follows.

\begin{conj}[Vojta's Conjecture \cite{Vojta}]{\label{conj:Vojta}}
	Let $V$ be a smooth projective variety over a number field $K$. Let $\mathcal{K}$ be the canonical divisor of $V$ and let $A$ be an ample normal crossings divisor, with associated height functions $h_\mathcal{K}$ and $h_A$. For any $\epsilon>0$, there is a proper Zariski-closed $X_\epsilon\subset V$ and a constant $C_\epsilon=C(V,K,A)$ such that \[h_\mathcal{K}(x)\le h_A(x)+C_\epsilon\] for all $x\in V(K)\backslash X_\epsilon(K)$.
\end{conj}

\begin{prop}{\label{prop:grandorbitcollision}}
	Let $K$ be a number field, and assume the $abc$-Conjecture for $K$. Let $f\in K[x]$ be a non-PCF, eventually stable polynomial of degree $3$ with distinct finite critical points $\gamma_1,\gamma_2$. If $f^i(\gamma_1)=f^j(\gamma_2)$ for some $i\ne j$, then $[\textup{Aut}(T_\infty):G(f)]~<~\infty$. 
\end{prop}

\begin{proof} As $f$ is non-PCF, the hypothesis implies that $\gamma_1$ and $\gamma_2$ have infinite forward orbit under $f$. If $f^i(\gamma_1)=f^j(\gamma_2)$ for $i<j$, then for all $n>2j-i$, any prime divisor of $f^n(\gamma_2)$ by definition is not a primitive prime divisor of $f^n(\gamma_1)$. The result then follows immediately from Theorem \ref{thm:PPDsufficiency} combined with Proposition \ref{prop: primitiveprime}. \end{proof}


\begin{prop}{\label{prop:nongeneric}}
	Let $K$ be a number field, and let $f\in K[x]$ be a non-PCF polynomial of degree $3$ with distinct finite critical points $\gamma_1,\gamma_2\in K$. Suppose there exists an $H\in\overline{\QQ}[x]$ such that $H\circ f^t=f^t\circ H$ for some $t\ge 1$, and that $H(0)=0$, and $H(f^m(\gamma_1))=f^n(\gamma_2)$ for some $m,n\ge 1$. Then $H=L\circ \psi^r$, $r\ge 0$, where $L$ is a linear polynomial commuting with some iterate of $f$, and $\psi^k=f^l$ for some $k,l\ge 0$. One of the following must hold: \begin{enumerate} 
		\item $L(x)=ax$ with $a\ne 1$, in which case $f$ is not eventually stable
		
		\item $L(x)=x$, in which case $f^i(\gamma_1)=f^j(\gamma_2)$ for some $i,j\ge 1$
		
		\item $L(0)\ne 0$, in which case, for any $\epsilon>0$, $h_\textup{gcd}(f^n(\gamma_1),f^n(\gamma_2))\le\epsilon d^n$ for all sufficiently large $n$.
		\end{enumerate}
\end{prop}
\begin{proof}
That there exist such an $L$ and $\psi$ follows from Ritt's Theorem \cite{Ritt2, Ritt1}. Cases (i) and (ii) are clear, so suppose we are in Case (iii). Conjugating $f$ to a monic centered representative $g$, so that $g=\mu^{-1}f\mu$, we obtain $H'=L'\circ \psi'^r$, where $H'=\mu^{-1}H\mu$, $L'=\mu^{-1}L\mu$, and $\psi'=\mu^{-1}\psi\mu$. As $g$ is in monic centered form, we must have $L'(x)=-x$ or $L'(x)=x$. If $L'(x)=x$, then we are in Case (ii); assume therefore that $L'(x)=-x$. If $\gamma_1'$ and $\gamma_2'$ are the critical points of $g$, then this implies $g^n(\gamma_1')=-g^n(-\gamma_1')=-g^n(\gamma_2')$. Noting that $\mu^{-1}(0)\ne 0$ (for otherwise, Case (i) or Case (ii) would hold), it follows that \[h_\textup{gcd}(g^n(\gamma_1')-\mu^{-1}(0),g^n(\gamma_2')-\mu^{-1}(0))=h_\textup{gcd}(f^n(\gamma_1),f^n(\gamma_2))\le\epsilon d^n\] for all sufficiently large $n$. 
\end{proof}
	
		
		

\begin{proof}[Proof of Theorem $\ref{thm:cubicpoly}$]
	The sufficiency of conditions (i) and (ii) follows from the proof of Conjecture 3.11 of \cite{jonessurvey}. To see that (iii) implies the index is infinite, note that \[\textup{Disc}(f^n)=\pm 3^{3^n}\textup{Disc}(f^{n-1})^3f^n(\gamma_1)f^n(\gamma_2).\] Base changing so that $K$ contains $i$ and $\sqrt{3}$, we get that $\textup{Disc}(f^n)$ is a square in $K_{n-1}$, so $\Gal(K_n/K_{n-1})$ contains only even permutations in its action on the roots of $f^n$. Since this holds for all $n\ge r$, $[\textup{Aut}(T_\infty):G(f)]=\infty$. 
	
	Conversely, assume $f$ is not PCF, that $f$ is eventually stable, and that $f$ has two distinct finite critical points $\gamma_1,\gamma_2\in K$ not satisfying (iii) for any $r$. If there is no polynomial $H$ as described in Theorem \ref{thm:Huang}, then Theorems \ref{thm:Huang} and \ref{thm:PPDsufficiency} combined with Proposition \ref{prop: primitiveprime} imply $[\textup{Aut}(T_{\infty}):G(f)]<\infty$ (note that we can conjugate by scaling to assume $f$ is monic, as in the statement of Theorem \ref{thm:PPDsufficiency}). On the other hand, if there does exist such an $H$, then by Proposition \ref{prop:nongeneric}, either $f^i(\gamma_1)=f^j(\gamma_2)$ for some $i,j\ge1$ with $i\ne j$, or, for any $\epsilon>0$, we have 
$$h_{gcd}(f^n(\gamma_1),f^n(\gamma_2))\le\epsilon d^n$$ 
for all sufficiently large $n$. In the former case, Proposition \ref{prop:grandorbitcollision} then yields $[\textup{Aut}(T_{\infty}):G(f)]<\infty$. In the latter case, Theorems \ref{thm:Huang} and \ref{thm:PPDsufficiency}, along with Proposition \ref{prop: primitiveprime} complete the proof. \end{proof}

\section{Example of families of rational maps with finite index}

\subsection{Cubic polynomial examples}

If we focus on specific examples of cubic and quadratic families of maps, we need not appeal to the $abc$-Conjecture or Vojta's Conjecture to get finite index results. The following proposition gives sufficient conditions we could check for a specific example or family of cubic polynomials to prove finite index arboreal representation.  

\begin{prop}\label{prop: cubicconditions}
Let $K$ be a number field and let $f(x)\in K[x]$ be a monic degree $3$ polynomial. Let $\gamma_1, \gamma_2$ denote the finite critical points of $f$. Suppose that $f^{n}(x)$ is irreducible and there exists a prime ${\mathfrak{p}}$ of $K$ such that 
\begin{itemize}
\item $v_{\mathfrak{p}}(f^n(\gamma_1))$ is odd 
\item $v_{\mathfrak{p}}(f^n(\gamma_2))=0$ 
\item $v_{\mathfrak{p}}(3)=0$
\item $v_{\mathfrak{p}}(f^i(\gamma_1))=v_{\mathfrak{p}}(f^i(\gamma_2))=0$ for $1\leq i<n$
\end{itemize}
then $[K_n:K_{n-1}]=6^{3^{n-1}}$.
\end{prop}

\begin{proof} 
Let $\alpha_1,\alpha_2, \dots, \alpha_{3^{n-1}}$ denote the roots of $f^{n-1}(x)$. As in Section  \ref{section:openimage}, let $M_i=K(f^{-1}(\alpha_i))$ and $\widehat{M_i}=K_{n-1}\prod_{i\neq j} M_j$. 

Since $v_{\mathfrak{p}}(f^n(\gamma_1))$ is odd, we must have $v_\mathfrak{p'}(\Disc(f(x)-\alpha_i)$ is odd for some prime $\mathfrak{p}'$ of $K(\alpha_i)$ lying over $\mathfrak{p}$. Thus, $\mathfrak{p}'$ must ramify in $M_i$, and $f(x)-\alpha_i$ must have at least a double root modulo $\mathfrak{p}'$. If this root had multiplicity 3, then $f'(x)$ has a double root modulo $\mathfrak{p}'$, so $\gamma_1\equiv\gamma_2\mod \mathfrak{p}'$ and hence $f^{n}(\gamma_1)\equiv f^{n}(\gamma_2) \mod \mathfrak{p}'$, contradicting our hypotheses. Hence, $f(x)-\alpha_i \equiv (x-\eta)^2(x-\xi) \mod \mathfrak{p}'$ and for any prime $\mathfrak{q}$ of $K(\alpha_i)$ lying over $\mathfrak{p}'$, $I(\mathfrak{q}|\mathfrak{p}')$ acts as a transposition on the roots of $f(x)-\alpha_i$. 

Following the arguments in the proof of Theorem~\ref{thm:PPDsufficiency}, we can see that $\Gal(K_n/\widehat{M_i})\cong S_3$ and hence $[K_n:K_{n-1}]=6^{3^{n-1}}$.
\end{proof}

Consider the cubic polynomials:
\[
g_1(z)=z^3- \frac{6012}{2755}z^2+\frac{12636}{13775}z+\frac{54}{95}
\]

and
\[
g_2(z)=z^3+7z^2-7.
\] 
In \cite{Combsthesis}, Combs shows that the image of the arboreal representation associated to $g_1$ is surjective. 
In \cite{Nicolespaper}, Looper shows that the image of the arboreal representation associated to $g_2$ is an index 2 subgroup of of $\textup{Aut}(T_\infty)$. Each of these examples can be shown to satisfy the hypotheses of Proposition \ref{prop: cubicconditions}.

\subsection{A family of degree $2$ rational maps}
In order to give an example of a family of quadratic rational map with surjective arboreal Galois representation, we follow the conventions of \cite{JonesManes2014}. To this end, we define the following notation:

\begin{tabular}{l l }
$f^n=\frac{P_n}{Q_n}$ & where $P_n$ and $Q_n$ are polynomials; \\
$\ell(r)$  &the leading coefficient of a polynomial $r$;\\ 
$v_{p}(n)$ & the $p$-adic valuation of $n$ for some prime $p$,\\
 & i.e. if $n = p^\nu d$ with $p \nmid d$, then $v_p(n) = \nu$; and\\
$\Res(Q,P)$ & the resultant of the polynomials $P$ and $Q$.
\end{tabular}

Proposition~\ref{prop: cubicconditions} is analogous to a result of Jones and Manes ~\cite[Corollary 3.8]{JonesManes2014}, which we state in the following theorem.


\begin{thm}[{\cite[Corollary 3.8]{JonesManes2014}}]\label{jonesmanescor}
Let $f=\frac{P_1(x)}{Q_1(x)}\in K(x)$ have degree $2$, let $c=Q_1P_1'-P_1Q_1',$ and suppose that $f^n(\infty)\neq 0$ for all $n\geq 1$ and that $f$ has two finite critical points $\gamma_1,$ $\gamma_2$ with $f(\gamma_i)\neq \infty$ for each $i.$ Suppose further that there exists a prime $\mathfrak{p}$ of $K$ with $v_{\mathfrak{p}}(P_n(\gamma_1)P_n(\gamma_2))$ odd and
\[
0=v_{\mathfrak{p}}(\ell(P_1))=v_{\mathfrak{p}}(\ell(c))=v_{\mathfrak{p}}(\Res(Q_1,P_1))=v_\mathfrak{p}(\Disc \ P_1)=v_\mathfrak{p}(P_j(\gamma_i))
\]

for $1\leq i\leq 2,$ $2\leq j \leq n-1.$ Then $[K_n:K_{n-1}]=2^{2^{n-1}}.$
\end{thm}

	Consider the family
	\[
	f_b(z)=\frac{z^2-2bz+1}{(-2+2b)z},
	\]
where $b\neq 1$ is an algebraic number.

    
   We will use Theorem~\ref{jonesmanescor} to show that an infinite number of members of this family have surjective arboreal representation over $\mathbb{Q}$. 
  This is the first example of an infinite family of non-polynomial rational maps having finite index arboreal representation. Prior to this work, there was a single non-polynomial quadratic  map that was known to have finite index; this example was given in~\cite{JonesManes2014}. 
  
  In the context of this particular family, we see that:
    
    \begin{align*}
    P_1(x)&=z^2-2bz+1,\\
    Q_1(x)&=(-2+2b)z,\\
    \ell(P_1)&=1,\\
    \Disc(P_1)&=4(b^2-1),\\
    c&=(-2+2b)z(2z-2b)-(z^2-2bz+1)(-2+2b)\\
    &=(-2+2b)(z^2-1), \\
	\ell(c)&=2(b-1),\\
    \Res(Q_1,P_1)&=4(b-1)^2,\\
    P_n(z)&=P_{n-1}(z)^2-2bP_{n-1}(z)Q_{n-1}(z)+Q_{n-1}(z)^2,\\
    Q_n(z)&=2(b-1)P_{n-1}(z)Q_{n-1}(z).
\end{align*}
   
	The point $\infty$ is a fixed point for this family. Since every member of this family maps $0$ to $\infty,$ we see that $0$ is always \emph{strictly} preperiodic. The critical points of this family are 1 and $-1$. The critical point $1$ is mapped by every member of the family to the critical point $-1,$ so this family's critical orbits collide. 
	    
	As long as the numerator of $f_b(z)$ is irreducible, the extension at level $n=1$ is maximal. Over $\QQ$, this is the case for any integer $b\neq 0,1$. We will show that, for all $n\geq 2$, the numerator $P_n(-1)$ has a primitive prime divisor with odd valuation that does not divide $2(b-1)(b+1)$.  
	
    

\begin{lemma}\label{L:divide denom forever}
Suppose $b\neq 1$ is an integer. Let $p$ be a prime and suppose that $p\mid Q_i(-1)$ for some $i\geq 1$. Then, for all $j>i$, we have that $p\mid Q_j(-1)$.
\end{lemma}
\begin{proof}
Suppose that $p\mid  Q_i(-1)$. Then, $$Q_{i+1}(-1)=2(b-1)P_i(-1)Q_i(-1),$$ so $p\mid Q_{i+1}~(-1)$. By induction, the result follows.
\end{proof}

\begin{lemma}\label{L:relatively prime}
Suppose $b\neq1$ is an integer. Then, for $n\geq1$, the only possible common divisors of $P_n(-1)$ and $Q_n(-1)$ are powers of $2$. 
\end{lemma}
\begin{proof}
We proceed by induction. For the base case, we remark that $P_1(-1)=2+2b$, and $Q_1(-1)=2-2b$. Thus, any common factor of $P_1(-1)$ and $Q_1(-1)$ must divide $P_1(-1)+Q_1(-1)=4,$ which gives the base case.

Suppose that the result is true for $P_{n}(-1)$ and $Q_{n}(-1)$. We wish to show that the same is true for $P_{n+1}(-1)$ and $Q_{n+1}(-1)$.

Suppose that $p$ is an odd prime dividing $Q_{n+1}(-1)$. Since 
\[Q_{n+1}(-1)=2(b-1)P_{n}(-1)Q_{n}(-1),\] 
it follows that $p\mid (b-1)$, or that $p\mid   Q_{n}(-1)$, or that $p\mid   P_{n}(-1)$. Moreover, if $p\mid (b-1)$, then $p\mid   Q_{n}(-1)$ by Lemma \ref{L:divide denom forever}. Thus,  either  $p\mid   Q_{n}(-1)$, or $p\mid   P_{n}(-1)$; by the inductive hypothesis, $p$ cannot divide both.

Suppose that  $p\mid   Q_{n}(-1)$ and $p\nmid   P_{n}(-1)$; the argument in the other case is identical.
Then, \[P_{n+1}(-1)=P_{n}(-1)^2-2bP_{n}(-1)Q_{n}(-1)+P_{n}(-1)^2\equiv P_{n}(-1)^2\not\equiv 0\mod{p},\]
so $p\nmid P_{n+1}(-1)$.
\end{proof}

\begin{lemma} \label{L:only in num once}
Suppose $b\neq 1$ is an integer. Let $p$ be an odd prime and suppose that $p\mid P_i(-1)$ for some $i\geq 1$. Then, for all $j>i$, we have that $p\mid Q_j(-1)$ and $p\nmid P_j(-1)$. 
\end{lemma}
\begin{proof}
Suppose $p$ is an odd prime and $p\mid P_i(-1)$. Since $$Q_{i+1}(-1)=2(b-1)P_i(-1)Q_i(-1),$$ it follows that $p\mid Q_{i+1}(-1)$. By Lemma \ref{L:divide denom forever}, it follows that $p\mid Q_j(-1)$ for all $j>i$. Since $p$ is an odd prime, it follows by Lemma \ref{L:relatively prime} that $p\nmid P_j(-1)$ for all $j>1$.
\end{proof}

 \begin{lemma}\label{lem: primprime}  Suppose $b\neq1$ is an integer. If $p$ is an odd prime divisor of $P_n(-1)$, then $p$ is a primitive prime divisor. Further, for any $n\geq 2$, the odd primes dividing $P_n(-1)$ do not divide $2(b-1)(b+1)$.
 \end{lemma}

\begin{proof}
Suppose that $p$ is an odd prime divisor of $P_n(-1)$. Suppose further that $p\mid P_i(-1)$ for some $i<n$. By Lemma \ref{L:only in num once}, it follows that $p\nmid P_j(-1)$ for all $j>i$; in particular, this implies that $p\nmid P_n(-1)$, a contradiction. Hence, any odd prime dividing $P_n(-1)$ is primitive. 

Note that $P_1(-1)=2+2b$, and $Q_1(-1)=2-2b$. Any odd prime $p$ dividing $2(b-1)(b+1)$ must therefore divide $P_1(-1)=2+2b$ or $Q_1(-1)=2-2b$. By Lemmas \ref{L:divide denom forever}, \ref{L:relatively prime}, and \ref{L:only in num once}, it follows that $p\nmid P_n(-1)$ for $n>1$.


 
	

\end{proof}

\begin{lemma}\label{lem: 2adicval} Let $b$ be an even integer. Then $v_2(P_n(-1))=v_2(Q_n(-1))=2^n-1$.
\end{lemma}

\begin{proof} We proceed by induction on $n$. Since $b$ is even $v_2(2+2b)=v_2(2-2b)=1$ so the result holds for $n=1$. Now suppose $v_2(P_{n-1}(-1))=v_2(Q_{n-1}(-1))=2^{n-1}-1$. Write $P_{n-1}=2^{2^{n-1}-1}u_{n-1}$ and $Q_{n-1}=2^{2^{n-1}-1}w_{n-1}$ where $u_{n-1}$ and $w_{n-1}$ are relatively prime odd integers. Then
\begin{align*}
Q_n(-1) &= 2(b-1)2^{2^{n-1}-1}u_{n-1}2^{2^{n-1}-1}w_{n-1}\\
&=2^{2^n-1}(b-1)u_{n-1}w_{n-1}
\end{align*}
and
\begin{align*}
P_n(-1)&=(2^{2^{n-1}-1}u_{n-1})^2-2b(2^{2^{n-1}-1}u_{n-1})(2^{2^{n-1}-1}w_{n-1})+(2^{2^{n-1}-1}w_{n-1})^2\\
&=2^{2^{n}-2}(u_{n-1}^2-2bu_{n-1}w_{n-1}+w_{n-1}^2).
\end{align*}
Since $u_{n-1}$ and $w_{n-1}$ are odd and $b$ is even, $u_{n-1}^2-2bu_{n-1}w_{n-1}+w_{n-1}^2\equiv 2 \mod 4$. Hence $v_2(P_{n})=v_2(Q_n)=2^n-1$ as desired. 
\end{proof}

\begin{lemma}\label{lem: bequiv4}  Let $b\equiv 2\mod 4$ and $b>0$. Then $P_n(-1)=2^{2^n-1}u_n$ where $u_n$ is odd and $u_n\neq \pm y^2$. 
 \end{lemma}
 
 \begin{proof}
 We can see by induction that $P_n(-1)>0$ and $Q_n(-1)<0$. Thus, $u_n>0$ so $u_n\neq -y^2$ for any integer $y$. By the proof of Lemma \ref{lem: 2adicval}, 
 \[2u_n=u_{n-1}^2-2bu_{n-1}w_{n-1}+w_{n-1}^2,\] where $u_{n-1}$ and $w_{n-1}$ are odd. Thus, $2u_n \equiv 6 \mod 8$. This implies $u_n\equiv 3$ or $7 \mod 8$ and hence $u_n\neq y^2$ for any integer $y$.

 \end{proof}
 
	
	
		 
	\begin{lemma}\label{lem: mod8}
		Suppose $b\equiv 4\mod{8}$. Then $P_n(-1)=2^{2^n-1}u_n$ where $u_n$ is odd and $u_n\neq \pm y^2$. 
	\end{lemma} 
	
	\begin{proof}
    
    We claim that $u_n$ and $w_n$ are each congruent to $\pm 3 \mod 8$. We proceed by induction, first note that $u_1 = 1+b$ and $w_1 = 1-b$, since $b\equiv 4 \mod 8$ the result holds for $n=1$. 
        
 Now suppose and $u_{n-1}$ and $w_{n-1}$ are both $\pm 3\mod 8$. Then, 

 \[w_n=(b-1)u_{n-1}w_{n-1}\equiv \pm 3^3\equiv \pm 11\mod 16,\]
 and
 \begin{align*}
 2u_n&=u_{n-1}^2-2bu_{n-1}w_{n-1}+w_{n-1}^2\\
 &\equiv 9+8+9\equiv 10\mod 16.
 \end{align*}
  So $u_n$, must be congruent to 5 or 13 modulo 16. Proving the claim that $u_n$ and $w_n$ are each congruent to $\pm 3 \mod 8$.
  This shows that $u_n\neq \pm y^2$ for any integer $y$.
	\end{proof}

\begin{proof}[Proof of Theorem~$\ref{thm: partialresults1}$.]
By design $f_b^n(\infty)\neq 0$ for all $n\geq 1$ and $f_b$ has two finite critical points $1$ and $-1$ whose orbits collide. That is, $f_b(1)=-1$. Using arguments similar to those in Lemma \ref{L:only in num once} and Lemma \ref{lem: primprime}, any common factor of $P_n(1)$ and $Q_n(1)$ must divide both $P_1(1)$ and $Q_1(1)$ and hence must divide $2(b-1)$. Since $f_b^n(1)=f_b^{n-1}(-1)$, we have 
\[\frac{P_n(1)}{Q_n(1)}=\frac{P_{n-1}(-1)}{Q_{n-1}(-1)}.\] 
Thus, by Lemmas~\ref{lem: primprime}, \ref{lem: bequiv4}, and \ref{lem: mod8}, for each $n\geq 2$, we have a prime $p$ satisfying $v_{p}(P_n(1)P_n(-1))=v_p(P_n(-1))$ is odd and \[
0=v_{p}(\ell(P_1))=v_{p}(\ell(c))=v_{p}(\Res(Q_1,P_1))=v_{p}(\Disc \ P_1)=v_{p}(P_j(\pm 1)), 
\] for all $2\leq j\leq n-1$. Thus we can apply Theorem~\ref{jonesmanescor} to conclude that $[K_n:K_{n-1}]=2^{2n-1}$ for each $n$, proving $[\Aut(T_\infty):~G(f_b)]=1$.
\end{proof}

\begin{remark}
This argument does not readily extend to $b\equiv 2^n\mod {2^{n+1}}$ for $n\geq 3.$ We note that the above proof gives us squares in the base case when $2^n+1=y^2,$ and the induction step gives us that $u,w\equiv 1\mod 2^{n+1}.$ A different argument may be needed for other values of $b.$
\end{remark}

\noindent\textbf{Acknowlegements.} This project began at the Women in Numbers 4 conference at BIRS. We would like to thank BIRS for hosting WIN4 and the AWM Advance Grant for funding the workshop. 

MM partially supported by the Simons Foundation grant \#359721.

\bibliographystyle{plain}	
\bibliography{Ref}

\end{document}